\documentclass[reqno,11pt]{amsart}
\usepackage[a4paper,margin=.96in,bottom=.98in,top=.8in]{geometry}
\usepackage{hyperref}
\usepackage{colonequals}
\usepackage{enumitem}
\usepackage{graphicx}
\usepackage{mathrsfs}

\newtheorem{theorem}{Theorem}[section]
\newtheorem{lemma}[theorem]{Lemma}

\theoremstyle{remark}
\newtheorem{remark}[theorem]{Remark}
\numberwithin{equation}{section}

\newcommand{\blank}{\,\cdot\,}
\DeclareMathOperator{\ran}{\mathrm{ran}}
\newcommand{\R}{\mathbb R}
\providecommand{\C}{\mathbb{C}} 
\renewcommand{\C}{\mathbb{C}}
\newcommand{\ina}{\textup{~in~}}
\newcommand{\ona}{\textup{~on~}}
\newcommand{\fora}{\textup{~for~}}
\newcommand{\ata}{\textup{~at~}}
\newcommand{\maps}{\colon}
\newcommand{\by}{\times}
\newcommand{\abs}[2][]{#1\lvert #2 #1\rvert} 
\newcommand{\norm}[2][]{#1\lVert #2 #1\rVert}
\newcommand{\grad}{\nabla}
\DeclareMathOperator{\img}{\mathrm{Im}}
\DeclareMathOperator{\re}{\mathrm{Re}}
\renewcommand{\Im}{\img}
\renewcommand{\Re}{\re}

\newcommand{\D}{{\mathbb D}}
\newcommand{\F}{\mathscr F}
\newcommand{\X}{\mathscr X}
\newcommand{\Y}{\mathscr Y}
\newcommand{\lin}{\mathscr L}
\newcommand{\wronsk}{\mathcal W}

\title[Non-symmetric solutions to an overdetermined problem]
{Non-symmetric solutions to an overdetermined problem for the Helmholtz equation in the plane}
\author{Miles H.~Wheeler}
\address{Department of Mathematical Sciences, University of Bath, Bath BA2 7AY, United Kingdom}
\email{mw2319@bath.ac.uk}
\date{August 30, 2025}
\begin{document}
\begin{abstract}
  In this note we construct smooth bounded domains $\Omega \subset
  \R^2$, other than disks, for which the overdetermined problem
  \begin{align*}
    \left\{
      \begin{alignedat}{2}
        \Delta u + \lambda u &= 0 &\qquad& \ina \Omega,\\
        u &= b &\qquad& \ona \partial \Omega,\\
        \frac{\partial u}{\partial n} &= c &\qquad& \ona \partial \Omega
      \end{alignedat}
    \right.
  \end{align*}
  has a solution for some constants $\lambda,b,c \ne 0$. These appear
  to be the first counterexamples to a conjecture of Willms and
  Gladwell \cite{wg:saddle}.
  %
  %
\end{abstract}
\maketitle

\section{Introduction}

In this introduction we consider overdetermined elliptic problems of
the form
\begin{subequations}\label{eqn:og}
  \begin{alignat}{2}
    \label{eqn:og:helm}
    \Delta u + \lambda u &= 0 &\qquad& \ina \Omega,\\
    \label{eqn:og:kin}
    u &= b  &\qquad& \ona \partial \Omega,\\
    \label{eqn:og:dyn}
    \frac{\partial u}{\partial n} &= c &\qquad& \ona \partial \Omega,
  \end{alignat}
\end{subequations}
where $\lambda,b,c$ are constants, $\Omega \subset \R^n$ is a smooth bounded domain, and
the function $u$ is non-constant. 
There are explicit radially symmetric solutions to such problems when
$\Omega$ is a ball, and it is natural to ask whether these are the only
solutions. 

The conjecture that balls are the only solutions of \eqref{eqn:og}
when $c=0$, attributed to Schiffer \cite[Problem~80]{yau:schiffer}, is
a famous and long-standing open problem in spectral geometry. Its
interest is due in part to its connection to the Pompeiu problem; see
\cite{bst:pompeiu,williams:pompeiu,berenstein:conjecture} and the
survey \cite{zalcman:survey}.
While the Schiffer conjecture remains open, there are many partial
results showing that it holds under additional assumptions,
including~\cite{berenstein:conjecture, by:inverse, aviles:symmetry,
deng:schiffer, kl:schiffer, mondal:schiffer}. In particular,
Agranovsky \cite{agranovsky:stability} (in two dimensions) and
Kobayashi \cite{kobayashi:perturbation} (in any dimension)
independently showed that the conjecture holds for small perturbations
of balls, and so there is no hope of finding counterexamples using
local bifurcation arguments. Also see \cite{canuto:stability} for a
related result which only requires \eqref{eqn:og:dyn} to hold in an
average sense.
%
In terms of negative results, Shklover~\cite{shklover:schiffer} showed
that a generalization of Schiffer's conjecture to Riemannian manifolds
is false, and recently Fall, Minlend, and Weth~\cite{fmw:schiffer}
found counterexamples on the sphere $\mathbb S^2$ -- disproving a
conjecture of Souam~\cite{souam:schiffer} -- as well as a family of
non-trivial solutions bifurcating from cylinders when $\Omega$ is
allowed to be unbounded. Even more recently, Enciso, Fernández, Ruiz,
and Sicbaldi~\cite{efrs:annuli} have constructed solutions bifurcating
from two-dimensional annuli with different values of the constant $b$
on each boundary component.

The analogue of Schiffer's conjecture for $b=0$ -- so that $\lambda$
is a Dirichlet rather than Neumann eigenvalue -- was posed by
Berenstein \cite{berenstein:conjecture} and is also still open.
Although less well-studied than the Neumann version, there are still
some partial results. For instance, the works
\cite{berenstein:conjecture, by:inverse} mentioned above contain
related results on both conjectures. As in the Neumann case, there are
also counterexamples to generalized versions of the conjecture.
Shklover~\cite{shklover:schiffer} constructed examples on certain
Riemannian manifolds, and Sicbaldi found positive solutions involving
unbounded domains bifurcating from cylinders \cite{sicbaldi:extremal};
also see \cite{ss:cylinder}. More recently,
Minlend~\cite{minlend:periodic} and Dai and
Zhang~\cite{dz:signchanging} have found sign-changing solutions
involving unbounded domains.

Several authors mention an analogue of Schiffer's conjecture without
any restrictions on the constants $b,c$, and to the best of our
knowledge this question has also remained open. The earliest statement
of the conjecture appears to be by Willms and Gladwell
\cite{wg:saddle}, who proved that it holds under the assumption that
$u$ has no saddle points. In a subsequent paper with
Chamberland~\cite{wcg:duality} they also discussed `dual' formulations
of the problem. Williams~\cite{williams:helmholtz} stated a more
precise version of the conjecture and studied (among other things) the
analyticity of solutions, while Souam \cite{souam:schiffer} proved an
analogue of the conjecture on the sphere $\mathbb S^2$ when $\lambda =
2$. Dalmasso~\cite{dalmasso:overdetermined} showed that, in two
dimensions and under some assumptions on the domain $\Omega$, for any
fixed $c \ne 0$ there can be at most finitely many pairs $(\lambda,b)$
for which \eqref{eqn:og} has a solution. In a later paper, Dalmasso
\cite{dalmasso:helmholtz} proved that the conjecture holds if either
$\lambda$ is at most the first Dirichlet eigenvalue of the Laplacian,
or if $\Omega$ is convex and symmetric about a hyperplane and $\lambda$
is at most the second Dirichlet eigenvalue. 
While \eqref{eqn:og} with both $b,c \ne 0$ is no longer an
overdetermined eigenvalue problem for the Laplacian, after making the
simple transformation $u \mapsto u+b$ it can be thought of as an
overdetermined semilinear eigenvalue problem
\begin{align}\label{eqn:serrin}
    \Delta u + \lambda f(u) = 0 \;\ina \Omega,
    \qquad 
    u = 0,\,  \frac{\partial u}{\partial n} = c \;\ona \partial \Omega,
\end{align}
with the specific nonlinearity $f(u)=u+b$. This form of the problem
was studied by Canuto and Rial~\cite{cr:overdetermined}, who showed
that the unit ball is an isolated solution provided $\lambda > 0$ lies
outside a certain countable set. Canuto~\cite{canuto:symmetry} later
proved a related result with a less restrictive hypothesis on
$\lambda$ but, at least in the context of \eqref{eqn:serrin}, a more
restrictive hypothesis on $c$. 

By Serrin's result \cite{serrin:symmetry}, semilinear problems of the
form \eqref{eqn:serrin} can have a solution $u$ with a strict sign
only when $\Omega$ is a ball. The existence of sign-changing solutions
for bounded domains $\Omega$ other than balls, however, has remained
open until quite recently. The only result we are aware of is due to
Ruiz~\cite{ruiz:signchanging}, who constructed sign-changing solutions
to \eqref{eqn:serrin} using nonlinearities of the form $f(u) =
u-(u^+)^3$ in dimensions $2$, $3$, and $4$. For sign-changing
solutions in unbounded domains, see the references
\cite{minlend:periodic, dz:signchanging} mentioned earlier. Lastly, we
mention the construction by Kamburov and Sciaraffia~\cite{ks:annular}
of solutions in annular domains with nonlinearity $f(u)=1$, where (as
in \cite{efrs:annuli}) they allow the constants in the boundary
conditions to differ on the two boundary components.

In this note, we use local bifurcation techniques to construct
families of solutions to \eqref{eqn:og} close to the unit disk
in two dimensions. The leading-order expansions for these solutions
imply that they are not disks, and so the above Willms--Gladwell
conjecture is false, at least in two dimensions. As a construction of
sign-changing solutions to an overdetermined semilinear problem
\eqref{eqn:serrin}, our proof is much simpler than that of
Ruiz~\cite{ruiz:signchanging}, which is to be expected given our much
simpler nonlinearity and our restriction to two dimensions. Also,
unlike in \cite{ruiz:signchanging} our functions $u$ are real-analytic.

\subsection{Statement of the main result}
Before stating our result more precisely, we need the following lemma about
Wronskians of Bessel functions. Here and in what follows we use the standard
notation $J_\nu$ for the Bessel function of the first kind with order
$\nu$.
\begin{lemma}\label{lem:bessel}
  For any integer $m \ge 4$, the Wronskian $\wronsk_{1,m} \colonequals J_1J_m' -
  J_m J_1'$ has a smallest positive root $\mu_m > 0$. Moreover, this
  root $\mu_m$ is simple, strictly decreases as a function 
  of $m$, and satisfies the inequalities
  $j_{1,1} < \mu_m < j_{0,2}$, where here $j_{1,1} \approx 3.8317$ is the first
  positive root of $J_1$ and $j_{0,2} \approx 5.5201$ is the second positive
  root of $J_0$.
\end{lemma}

\begin{theorem}\label{thm:main}
  Fix $n=2$, $b=1$, and $\alpha \in (0,1)$.
  For any integer $m \ge 4$, there exists $\varepsilon_0 > 0$
  and a curve of classical solutions to \eqref{eqn:og}, parametrized by
  $\varepsilon \in (-\varepsilon_0,\varepsilon_0)$, with the following properties.
  \begin{enumerate}[label=\rm(\roman*)]
  \item \label{thm:main:notdisk} For $\varepsilon \ne 0$, the domain $\Omega(\varepsilon)$ is not a disk.
  \item \label{thm:main:mfold} The solutions are $m$-fold symmetric, in the sense that
    $\Omega(\varepsilon),u(\varepsilon)$ are invariant under rotations by an angle $2\pi/m$ as
    well as reflections across the horizontal axis.
  \item The domains $\Omega(\varepsilon)$ are described by conformal mappings $\phi(\varepsilon)
    \maps \D \to \Omega(\varepsilon)$, where $\D$ is the unit disk, and the unknowns 
    \begin{align}
      \label{eqn:spaces}
      (u \circ \phi,\phi,c,\lambda) \in C^{2+\alpha}(\overline\D) \by C^{2+\alpha}(\overline\D,\C) \by \R^2
    \end{align}
    depend real-analytically on $\varepsilon$. 
  \item \label{thm:main:asym} As $\varepsilon \to 0$ we have the asymptotic expansions
    \begin{subequations}\label{eqn:asym}
      \begin{align}
        \label{eqn:asym:F}
        \phi(re^{i\theta};\varepsilon) 
        &= re^{i\theta}  + \varepsilon (re^{i\theta})^{m+1} + O(\varepsilon^2),\\
        \label{eqn:asym:u}
        (u \circ \phi)(re^{i\theta};\varepsilon) &= 
        \frac{J_0(\mu_m r)}{J_0(\mu_m)}
        +
        \varepsilon\mu_m
        \bigg(
        \frac{J_1(\mu_m)J_m(\mu_m r)}{J_0(\mu_m)J_m(\mu_m)}
        -\frac{ J_1(\mu_mr)}{J_0(\mu_m)}r^{m+1}
        \bigg)
        \cos m\theta
        + O(\varepsilon^2),\\
        \label{eqn:asym:c}
        c(\varepsilon) &= -\mu_m\frac{J_1(\mu_m)}{J_0(\mu_m)} + O(\varepsilon^2),\\
        \label{eqn:asym:lam}
        \lambda(\varepsilon) &= \mu_m^2 + O(\varepsilon^2)
      \end{align}
    \end{subequations}
    in the spaces \eqref{eqn:spaces}. 
  \end{enumerate}
\end{theorem}
\begin{figure}
  \includegraphics[scale=1.1]{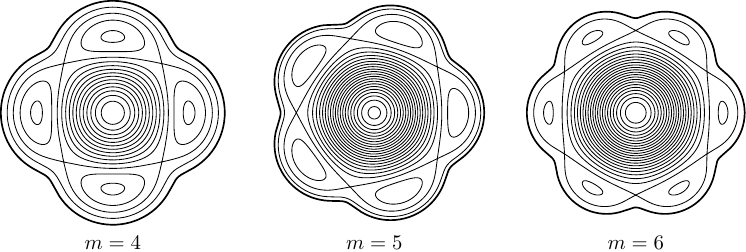}
  \caption{Exaggerated sketches of the domains $\Omega(\varepsilon)$ and functions
  $u(\varepsilon)$ from Theorem~\ref{thm:main} for $m=4,5,6$, based on the
  leading-order approximation \eqref{eqn:asym}. \label{fig:mfold}}
\end{figure}
\begin{remark}\label{rk:regularity}
  For each fixed $\varepsilon$, the domains $\Omega(\varepsilon)$ 
  are in fact real-analytic. As observed by
  Williams~\cite{williams:helmholtz}, this follows from
  \cite[Theorem~2]{kn:regularity}. The functions
  $u(\blank;\varepsilon)$ are therefore also real-analytic up to the
  boundary.
\end{remark}
\begin{remark}
  Possibly after shrinking $\varepsilon_0 > 0$,
  the first conclusion \ref{thm:main:notdisk} of
  Theorem~\ref{thm:main} is an immediate consequence of
  \eqref{eqn:asym:F}. The
  asymptotic formulas \eqref{eqn:asym} are illustrated in
  Figure~\ref{fig:mfold} for $m=4,5,6$, using an unreasonably large
  value of $\varepsilon > 0$ so that the saddles and local extrema of $u$ are
  more apparent. Note that the existence of at least one saddle point
  is guaranteed by \cite[Theorem~1]{wg:saddle}. While the arrangement
  of critical points in the figure is qualitatively correct, the
  domains $\Omega(\varepsilon)$ of the exact solutions with small $\varepsilon$ will
  necessarily be convex.
\end{remark}
\begin{remark}\label{rk:scaling}
  If $(\Omega,u)$ solves \eqref{eqn:og} with parameters
  $(\lambda,b,c)$, then, for any constants $A > 0$ and $B \in \R$, so does $(A^{-1} \Omega, B
  u(A\blank))$ with parameters $(A^2\lambda,B b,A B c)$. Thus, provided $b \ne
  0$, we are free to set $b=1$ in Theorem~\ref{thm:main} without loss
  of generality. Similarly, we are free to fix one of the other
  parameters $\lambda,c$ (provided they are nonzero) or alternatively to fix
  the scale of $\Omega$. We choose the latter, normalizing the conformal
  mapping to satisfy $\phi'(0)=1$; see \eqref{eqn:killscaling} below. 
\end{remark}
\begin{remark}\label{rk:bessel}
  As we will see in Section~\ref{sec:bessel}, the bounds on $\mu_m$ in
  Lemma~\ref{lem:bessel} imply that $J_0(\mu_m)$ and $J_1(\mu_m)$ are
  negative while $J_m(\mu_m)$ is positive. Thus the denominators
  appearing in \eqref{eqn:asym} are nonzero, the $O(\varepsilon)$ terms in
  \eqref{eqn:asym:F} do not vanish identically, and \eqref{eqn:asym:c}
  implies $c(\varepsilon) < 0$ for $\abs \varepsilon$ sufficiently small. One can also
  check that, as must be the case, $\mu_m$ lies in the discrete set $\Lambda$
  defined in \cite[Lemma~3.7 and Definition~3.9]{cr:overdetermined}
  where their local uniqueness proof fails.
\end{remark}

\subsection{Outline of the paper} 
In Section~\ref{sec:reform}, we transform \eqref{eqn:og} into an
abstract operator equation $\F(v,w,\gamma;\mu)=0$ where $\F \maps \X
\by I \to \Y$ is a real-analytic mapping between Banach
spaces. There are many ways to effect such a transformation; we use
conformal mappings combined with a change of dependent variables which
partially decouples the linearized equations. In
Section~\ref{sec:proof}, we prove Theorem~\ref{thm:main}, assuming
that Lemma~\ref{lem:bessel} holds, by applying a real-analytic version
of the standard Crandall--Rabinowitz theorem \cite{rabinowitz:simple}.
In part because of the careful choice of variables in
Section~\ref{sec:reform}, the required analysis of the linearized
operators $\lin(\mu)$ quickly boils down to questions about the
Wronskian $\wronsk_{1,m}$ appearing in Lemma~\ref{lem:bessel}, which
we then finally prove in Section~\ref{sec:bessel} using ideas from
\cite{palmai:interlacing}.

\section{Reformulation}\label{sec:reform}
As is usual in bifurcation analyses of free boundary problems, we
begin by reducing \eqref{eqn:og} to a problem in a fixed domain by
introducing an appropriate diffeomorphism which becomes a new unknown.
Since we are in two dimensions, conformal mappings are a natural
option which simplify some of the calculations, but this choice is not
essential. To streamline the linear analysis, we then make a further
change of dependent variables which mixes the unknown $u$ and this
conformal mapping.

\subsection{Preliminaries}
To avoid dealing with square roots in some of the calculations below,
we replace the boundary condition \eqref{eqn:og:dyn} with half of its
square,
\begin{align}
  \label{eqn:og:dyn2}
  \frac 12 \left(\frac{\partial u}{\partial n}\right)^2 = \frac 12 c^2 \qquad \ona \partial \Omega.
\end{align}
This is of course equivalent provided we choose the sign of $c$
appropriately. For similar reasons, we switch from the parameter $\lambda > 0$
to
\begin{align}
  \label{eqn:mudef}
  \mu \colonequals \sqrt \lambda.
\end{align}
We will assume throughout that $\mu$ lies 
in the interval 
\begin{align}
  \label{eqn:I}
  I \colonequals (j_{1,1},j_{0,2}),
\end{align}
where, as in Lemma~\ref{lem:bessel}, $j_{1,1} \approx 3.8317$ is the first
positive root of $J_1$ and $j_{0,2} \approx 5.5201$ is the second positive
root of $J_0$. 

\subsection{Conformal change of variables}

\begin{subequations}\label{eqn:conf}
  Fixing once and for all a Hölder parameter $\alpha \in (0,1)$, suppose that
  $w \in C^{2+\alpha}(\overline\D,\R^2)$ is a holomorphic vector field on the
  unit disk $\D \subset \R^2 \cong \C$, i.e.~its components $(w_1,w_2)$ satisfy the
  Cauchy--Riemann equations
  \begin{align}
    \label{eqn:conf:holo1}
    \partial_1 w_1 - \partial_2 w_2 &= 0,\\
    \label{eqn:conf:holo2}
    \partial_2 w_1 + \partial_1 w_2 &= 0
  \end{align}
  in $\D$. Provided $\norm w_{C^{2+\alpha}(\D)}$ is sufficiently small, the
  near-identity map $\phi = \mathrm{id} + w$ is a diffeomorphism onto its
  image, which is a $C^{2+\alpha}$ domain $\Omega$. Letting $\tilde u = u \circ \phi \in
  C^{2+\alpha}(\overline\D)$, it is straightforward to check that the
  problem \eqref{eqn:og:helm}, \eqref{eqn:og:kin}, \eqref{eqn:og:dyn2}
  with $b=1$ and $\lambda=\mu^2$ is then equivalent to
  \begin{alignat}{2}
    \label{eqn:conf:bulk}
    \Delta \tilde u +  \mu^2\det(I+Dw)\, \tilde u &= 0 &\qquad& \ina \D, \\
    \label{eqn:conf:kin}
    \tilde u &= 1&\qquad& \ona \partial\D, \\
    \label{eqn:conf:dyn}
    \tfrac 12 \tilde u_r^2 - \tfrac 12 c^2 \det(I+Dw) &= 0 &\qquad& \ona \partial \D,
  \end{alignat}
\end{subequations}
where a subscript $r$ is shorthand for an application of the usual
radial derivative $\partial_r = \abs x^{-1} x \cdot \grad$. For algebraic convenience
we have also multiplied both \eqref{eqn:og:helm} and
\eqref{eqn:og:dyn2} by the positive function $\det(I+Dw)$. Identifying $\R^2$ with $\C$, we
note that $\det(I+Dw) = \abs{1+w'}^2$ where $w'$ is the complex
derivative of the holomorphic function $w$.

The transformed problem \eqref{eqn:conf} has a family of ``trivial''
radially symmetric solutions, expressed in polar coordinates
$(r,\theta)$ as
\begin{align}
  \label{eqn:triv}
  \tilde u = U(r;\mu) \colonequals \frac{J_0(\mu r)}{J_0(\mu)},
  \qquad 
  w=0,
  \qquad 
  c = U_r(1;\mu) = 
  \frac{\mu J_0'(\mu)}{J_0(\mu)}
  =-\frac{\mu J_1(\mu)}{J_0(\mu)},
\end{align}
where in the last equality we have used the fact that $J_0'=-J_1$.
We recall that, for any $k \ge 0$, the Bessel function $J_k$
is, up to scaling, the unique solution of the ordinary
differential equation
\begin{align}
  \label{eqn:besselode}
  J_k'' + \frac 1\mu J_k' + \left(1-\frac{k^2}{\mu^2}\right)J_k &= 0
  \qquad \fora \mu > 0
\end{align}
which is finite at $\mu = 0$. We are interested in the solutions
\eqref{eqn:triv} only for $\mu$ lying in the interval $I$ from \eqref{eqn:I},
and on this interval both $J_0$ and $J_1$ are strictly negative.

\subsection{Simplifying the linear part}
A downside of the formulation \eqref{eqn:conf} is that the unknowns
$u,w$ are coupled in both \eqref{eqn:conf:bulk} and
\eqref{eqn:conf:dyn}. At the nonlinear level this is unavoidable, but
at the linear level the equations can be partially decoupled using a
standard trick motivated by the chain rule.

Define $v \in C^{2+\alpha}(\overline\D)$ and $\gamma \in \R$ in terms of $\tilde
u,w,c$ by
\begin{align}
  \label{eqn:vgamma}
  \begin{aligned}
    \tilde u &\equalscolon U + v + \grad U \cdot w, \\
    c &\equalscolon U_r(1;\mu) - \gamma.
  \end{aligned}
\end{align}
Inserting \eqref{eqn:vgamma} into \eqref{eqn:conf} and grouping the
linear terms in $v,w,\gamma$, there are several cancellations and we are
left with the system
\begin{subequations}\label{eqn:newform}
  \begin{alignat}{2}
    \label{eqn:newform:holo1}
    \partial_1 w_1 - \partial_2 w_2 &= 0 &\qquad& \ina \D,\\
    \label{eqn:newform:holo2}
    \partial_2 w_1 + \partial_1 w_2 &= 0&\qquad& \ina \D,\\
    \label{eqn:newform:bulk}
    \Delta v +  \mu^2v &= N_1(\grad v,Dw,\gamma,x;\mu)  &\qquad& \ina \D, \\
    \label{eqn:newform:kin}
    v + U_r\, x \cdot w &= 0 &\qquad& \ona \partial\D, \\
    \label{eqn:newform:dyn}
    U_r v_r - U_{rr} v + U_r \gamma &= N_2(\grad v,D w,\gamma,x;\mu) &\qquad& \ona \partial \D,
  \end{alignat}
\end{subequations}
where as a final step we have used \eqref{eqn:newform:kin} to
eliminate $w$ from the left hand side of \eqref{eqn:newform:dyn}. 
The nonlinear terms are given explicitly by
\begin{align*}
  N_1(\grad v, Dw,\gamma,x;\mu)
  &= -\mu^2\big(U \det Dw + (v + \grad U \cdot w) \grad \cdot w 
  +(v + \grad U \cdot w) \det Dw\big), \\
  N_2(\grad v, Dw,\gamma,x;\mu)
  &=
  \tfrac 12 (v_r^2 - \gamma^2)
  + U_r (\tfrac 12v_r+\gamma) \grad \cdot w
  + U_{rr} v_r x \cdot w
  \\&\qquad
  + \tfrac 12(\grad U \cdot w)_r^2
  - \tfrac 12 U_r^2 \det Dw
  + \gamma U_r \det Dw
  - \tfrac 12 \gamma^2 \grad \cdot w,
\end{align*}
although we emphasize that these formulas play essentially no
role in the following analysis. 

\subsection{Symmetry and functional setting}
We restrict attention to solutions of \eqref{eqn:newform} which are
$m$-fold symmetric in the sense of Theorem~\ref{thm:main}\ref{thm:main:mfold}.
Identifying $\R^2$ with $\C$, this can be conveniently
expressed as
\begin{subequations}\label{eqn:mfold}
  \begin{gather}
    \label{eqn:mfold:v}
    v(e^{2\pi i/m}z) = v(z) = v(\bar z),\\
    \label{eqn:mfold:w}
    w(e^{2\pi i/m}z) = e^{2\pi i/m} w(z),
    \qquad 
    w(\bar z) = \overline{w(z)}
  \end{gather}
\end{subequations}
for all $z \in \overline\D$. We also require $Dw(0)=0$, which by
\eqref{eqn:newform:holo1}, \eqref{eqn:newform:holo2}, and
\eqref{eqn:mfold:w} is equivalent to the single real condition
\begin{align}
  \label{eqn:killscaling}
  \partial_1 w_1(0) = 0.
\end{align}
This enforces the normalization $\phi'(0)=1$ for the conformal mapping
$\phi=\mathrm{id}+w$, which in turn fixes the scale of $\Omega$; see 
Remark~\ref{rk:scaling}.

We now introduce the Banach spaces
\begin{align*}
  \X &\colonequals \big\{ (v,w,\gamma) \in C^{2+\alpha}(\overline \D) \by C^{2+\alpha}(\overline \D, \C)
  \by \R : \text{\eqref{eqn:newform:holo1}, \eqref{eqn:newform:holo2}, \eqref{eqn:mfold}, \eqref{eqn:killscaling} hold}\big\},\\
  \Y &\colonequals \big\{ (f_1,f_2,f_3) \in C^\alpha(\overline \D) \by C^{2+\alpha}(\partial\D) \by C^{1+\alpha}(\partial\D) :
  \text{each $f_i$ has the symmetry \eqref{eqn:mfold:v}} \big\}.
\end{align*}
Defining the interval $I=(j_{1,1},j_{0,2})$ as in \eqref{eqn:I}, we
can then interpret \eqref{eqn:newform} as an operator equation
$\F(v,w,\gamma;\mu) = 0$, where $\F \maps \X \by I \to \Y$ is the real-analytic
mapping given by
\begin{align*}
  \F_1(v,w,\gamma;\mu) &\colonequals \Delta v +  \mu^2v - N_1(\grad v,Dw,\gamma,x;\mu), \\
  \F_2(v,w,\gamma;\mu) &\colonequals v + U_r\, x \cdot w, \\
  \F_3(v,w,\gamma;\mu) &\colonequals U_r v_r - U_{rr} v + U_r \gamma - N_2(\grad v,D w,\gamma,x;\mu).
\end{align*}
Since $\F$ only involves partial derivatives and compositions with
real-analytic functions, its real-analyticity is standard; see for
instance \cite{lo:composition} and \cite[proof of
Theorem~II.5.2]{valent:elasticity}. It is also not difficult to verify
that $\F$ respects the symmetries in the definitions of $\X,\Y$,
especially if one uses \eqref{eqn:vgamma} to reintroduce $\tilde u$ and
writes $\det(I+Dw)=\abs{1+w'}^2$ using complex variables.

The trivial solutions \eqref{eqn:triv} are now represented by points
$(0,\mu) \in \X \by I$, where the associated linearized operators 
\begin{align*}
  \lin(\mu) \colonequals D_{(v,w,\gamma)} \F(0,0,0;\mu) \maps \X \longrightarrow \Y 
\end{align*}
are given in components by 
\begin{align}
  \label{eqn:lin}
  \lin(\mu) 
  \begin{pmatrix}
     v\\ w \\  \gamma
  \end{pmatrix}
  =
  \begin{pmatrix}
    \Delta  v +  \mu^2  v \\
     v + U_r\, x \cdot  w \\
    U_r v_r  -  U_{rr}  v + U_r  \gamma
  \end{pmatrix}.
\end{align}
The advantage of \eqref{eqn:newform} over a more direct approach to
\eqref{eqn:conf} is that the first and third components of
\eqref{eqn:lin} have constant coefficients and do not involve $w$.

\section{Proof of the theorem}\label{sec:proof}

In this section we prove Theorem~\ref{thm:main},
assuming Lemma~\ref{lem:bessel}. The main tool is the following
real-analytic version of the celebrated Crandall--Rabinowitz
theorem~\cite{rabinowitz:simple} on bifurcation from a simple
eigenvalue. 
\begin{theorem}[Theorem 8.3.1 in \cite{bt:analytic}]\label{thm:cr}
  Let $X,Y$ be Banach spaces, let $I \subset \R$ be an open
  interval, and let $F \maps X \by I \to Y$ a real-analytic mapping with 
  $F(0,\mu) = 0$ for all $\mu \in I$. Denoting the associated
  linearized operators by $L(\mu) \colonequals D_x F(0,\mu)$, suppose that,
  for some $\mu^* \in I$,
  \begin{enumerate}[label=\rm(\alph*)]
  \item \label{cr:fred} $L(\mu^*)$ is Fredholm with index 0;
  \item \label{cr:ker} the kernel of $L(\mu^*)$ is one dimensional, spanned by $\xi^*
    \in X$; and
  \item \label{cr:trans} \textup{(Transversality)} $\partial_\mu L(\mu^*) \xi^* \notin \ran L(\mu^*)$.
  \end{enumerate}
  Then there exists $\varepsilon_0 > 0$ and a pair of analytic functions
  $(\tilde x,\tilde \mu) \maps (-\varepsilon_0,\varepsilon_0) \to X \by I$ such that
  \begin{enumerate}[label=\rm(\roman*)]
  \item $F(\tilde x(\varepsilon),\tilde \mu(\varepsilon)) = 0$ for all $\varepsilon \in (-\varepsilon_0,\varepsilon_0)$;
  \item $\tilde x(0)=0$, $\tilde \mu(0)=\mu^*$, and $\tilde x_\varepsilon(0) = \xi^*$;
    and
  \item there exists an open neighborhood $U$ of $(0,\mu^*)$ in $X \by \R$
    such that 
  \begin{align*}
    \big\{ (x,\mu) \in U : F(x,\mu)=0,\, x \ne 0 \big\}
    =
    \big\{ (\tilde x(\varepsilon),\tilde \mu(\varepsilon))  : 0 < \abs \varepsilon < \varepsilon_0 \big\}.
  \end{align*}
  \end{enumerate}
\end{theorem}
    
We will verify the hypotheses of Theorem~\ref{thm:cr} in a series of
lemmas. First, we state a basic result about the second component
of the operator $\lin(\mu)$.
\begin{lemma}\label{lem:schwarz}
  Suppose that $g \in C^{2+\alpha}(\partial \D)$ has the symmetries
  \eqref{eqn:mfold:v}. Then the problem
  \begin{align}\label{eqn:schwarz}
    \begin{alignedat}{2}
      \partial_1 w_1 - \partial_2 w_2 &= 0 &\quad& \ina \D,\\
      \partial_2 w_1 + \partial_1 w_2 &= 0 &\quad& \ina \D,\\
      x \cdot  w &= g &\quad& \ona \partial\D
    \end{alignedat}
  \end{align}
  has a solution $\smash{w \in C^{2+\alpha}(\overline \D,\R^2)}$ satisfying
  \eqref{eqn:mfold:w} and \eqref{eqn:killscaling} if and only if $g$
  has mean zero, and in this case the solution is unique.
  \begin{proof}
    Throughout the proof we identify $\R^2$ with $\C$ whenever
    convenient. Given $g$ as in the statement, let $G \in
    C^{2+\alpha}(\overline \D,\C)$ be the unique solution to the Schwarz
    boundary value problem
    \begin{align}\label{eqn:schwarztilde}
      \begin{alignedat}{2}
        \partial_1G_1 - \partial_2 G_2 &= 0 &\quad& \ina \D,\\
        \partial_2G_1 + \partial_1 G_2 &= 0 &\quad& \ina \D,\\
        G_1 = \Re G &= g &\quad& \ona \partial\D
      \end{alignedat}
    \end{align}
    with $\Im G(0) = 0$. By uniqueness and the symmetry of $g$, we
    deduce that $G$ has the symmetry \eqref{eqn:mfold:v}. If $g$ has
    zero mean, then the mean value property for the harmonic function
    $\Re G$ implies that $\Re G(0) = 0$ and hence $G(0)=0$. It is now
    straightforward to check that $w(z) \colonequals zG(z)$ solves
    \eqref{eqn:schwarz} and satisfies \eqref{eqn:mfold:w} and
    \eqref{eqn:killscaling}.
    Conversely, if $w$ satisfies \eqref{eqn:schwarz} as well as
    \eqref{eqn:mfold:w} and \eqref{eqn:killscaling}, then
    $G(z) \colonequals w(z)/z$ is a well-defined holomorphic function on $\D$
    which solves \eqref{eqn:schwarztilde} with $\Im G(0)=0$.
    Uniqueness for \eqref{eqn:schwarztilde} now gives uniqueness for
    $w$, while the mean value property for $\Re G$ and $\Re G(0) = \partial_1
    w_1(0) = 0$ together imply that $g$ has zero mean.
  \end{proof}
\end{lemma}


\begin{lemma}\label{lem:index}
  For any $\mu \in I$, the operator $\lin(\mu)$ is Fredholm with index 0.
  \begin{proof}
    As $\lin(\mu)$ is a compact perturbation of the 
    operator
    \begin{align}
      \label{eqn:tildelindef}
      \tilde\lin(\mu) \maps 
      \begin{pmatrix}
         v\\ w \\  \gamma
      \end{pmatrix}
      \mapsto 
      \begin{pmatrix}
        \Delta  v - v \\
        v + U_r\, x \cdot  w \\
        U_r v_r
      \end{pmatrix},
    \end{align}
    it suffices to show that $\tilde \lin(\mu)$ is Fredholm with index 0.
    To this end, let $f = (f_1,f_2,f_3) \in \Y$. The first and third components of the
    equation 
    \begin{align}
      \label{eqn:tildelin}
      \tilde \lin(\mu) (v,w,\gamma) = (f_1,f_2,f_3)
    \end{align}
    form an
    inhomogeneous Neumann problem
    \begin{align*}
      \begin{alignedat}{2}
        \Delta  v - v &= f_1  &\quad& \ina \D,\\
        v_r &= U_r^{-1} f_3 &\quad& \ona \partial\D. 
      \end{alignedat}
    \end{align*}
    This problem has a unique solution $v = \tilde v(f_1,f_3) \in
    C^{2+\alpha}(\overline\D)$, and uniqueness together with the
    symmetries of the data $f_1,f_3$ in the definition of $\Y$ forces
    $v$ to satisfy \eqref{eqn:mfold:v}.  

    It remains to consider the second component of
    \eqref{eqn:tildelin}, which rearranges to
    \begin{align*}
      x \cdot  w = U_r^{-1}(f_2 - \tilde v(f_1,f_3)) \qquad \ona \partial \D.
    \end{align*}
    Applying Lemma~\ref{lem:schwarz}, we deduce that
    \eqref{eqn:tildelin} has a solution if and only if the
    solvability condition
    \begin{align*}
      \int_{\partial\D} (f_2 - \tilde v(f_1,f_3))\, ds = 0
    \end{align*}
    holds, and that in this case the solution is unique. Since the
    parameter $\gamma \in \R$ does not appear on the right hand side of
    \eqref{eqn:tildelindef}, we conclude that $\tilde \lin(\mu)$ has index
    zero and the proof is complete.
  \end{proof}
\end{lemma}

\begin{lemma}\label{lem:ker}
  Let $\mu_m$ be as in Lemma~\ref{lem:bessel}. Then
  the kernel $\lin(\mu_m)$ is one-dimensional, spanned by
  $(V_m(\blank;\mu_m),W_m(\blank;\mu_m),0)$ where $V_k,W_k$ are given by the
  formulas
  \begin{align}
    \label{eqn:VW}
    \begin{aligned}
      V_k(re^{i\theta};\mu) &\colonequals J_k(\mu r) \cos k\theta, \\
      W_k(re^{i\theta};\mu) &\colonequals \frac{J_k(\mu)J_0(\mu)}{\mu J_1(\mu)} r^{k+1}
      \begin{pmatrix}
        \cos((k+1)\theta) \\
        \sin((k+1)\theta)
      \end{pmatrix}.
    \end{aligned}
  \end{align}
  \begin{proof}
    Let $\mu \in I$ and suppose that $(v,w,\gamma) \in \X$ lies in the kernel of $\lin(\mu)$, i.e.
    \begin{subequations}\label{eqn:ker}
      \begin{alignat}{2}
        \label{eqn:ker:bulk}
        \Delta v +  \mu^2v &= 0  &\qquad& \ina \D, \\
        \label{eqn:ker:kin}
        v + U_r\, x \cdot w &= 0 &\qquad& \ona \partial\D, \\
        \label{eqn:ker:dyn}
        U_r v_r - U_{rr} v + U_r \gamma &= 0 &\qquad& \ona \partial \D.
      \end{alignat}
    \end{subequations}
    If $v \equiv 0$, then Lemma~\ref{lem:schwarz} and \eqref{eqn:ker:dyn}
    immediately imply that $w\equiv 0$ and $\gamma=0$ as well. 
   
    Suppose then that $v \not \equiv 0$. From \eqref{eqn:ker:bulk} and
    familiar Fourier series arguments, it is enough to consider the case
    where $v=V_k$ for some $k \ge 0$ which is an integer multiple of
    $m$. If $k=0$, then $v = J_0(\mu_m) < 0$ on $\partial\D$, and so
    \eqref{eqn:ker:kin} violates the solvability condition in
    Lemma~\ref{lem:schwarz}. Thus we can restrict our attention to $k \ge
    1$. By Lemma~\ref{lem:schwarz}, the boundary condition
    \eqref{eqn:ker:kin} can then be uniquely solved for $w$, and an
    easy calculation confirms that this solution is $w=W_k$. Averaging
    \eqref{eqn:ker:dyn} yields $\gamma = 0$.

    Substituting $v=V_k$, $w=W_k$, and $\gamma=0$ in
    \eqref{eqn:ker:dyn} and recalling \eqref{eqn:triv}, we are finally
    left with
    \begin{align*}
      U_r V_r - U_{rr} V
      &=
      -\frac {\mu^2}{J_0(\mu)}
      \wronsk_{1,k}(\mu)
      \cos k\theta = 0
      \quad \ona \partial\D,
    \end{align*}
    where $\wronsk_{1,k} = J_1J_k'  - J_k J_1'$ is the Wronskian from
    Lemma~\ref{lem:bessel}. Setting $\mu=\mu_m$, we have from
    the same lemma that $\wronsk_{1,m}(\mu_m) = 0$, and moreover
    that $\wronsk_{1,k}(\mu_m) < 0$ for all integers $k > m$. The proof
    is complete.
  \end{proof}
\end{lemma}

\begin{lemma}\label{lem:transverse}  
  There are no solutions $(v,w,\gamma) \in \X$ to the equation
  \begin{align}
    \label{eqn:nontrans}
    \lin(\mu_m) (v,w,\gamma) + \partial_\mu \lin(\mu_m) (V_m(\blank;\mu_m),W_m(\blank;\mu_m),0) = 0.
  \end{align}
  \begin{proof}
    To simplify the notation, we suppress the first
    arguments of the functions $V_m,W_m$ defined in \eqref{eqn:VW}.
    We have seen in the proof of Lemma~\ref{lem:ker} that, for any
    $\mu \in I$, $(V_m(\mu),W_m(\mu),0)$ satisfies
    \begin{align*}
      \lin(\mu) 
      \begin{pmatrix}
        V_m(\mu) \\ W_m(\mu) \\ 0 
      \end{pmatrix}
      =
      -\frac {\mu^2}{J_0(\mu)} \wronsk_{1,m}(\mu)
      \begin{pmatrix}
        0 \\ 0 \\ 
        \cos m\theta
      \end{pmatrix}.
    \end{align*}
    Differentiating this identity with respect to $\mu$ and using the fact that
    $\wronsk_{1,m}(\mu_m)=0$, we find 
    \begin{align}
      \label{eqn:aha}
      \lin(\mu) 
      \begin{pmatrix}
        \partial_\mu V_m(\mu_m) \\ \partial_\mu W_m(\mu_m) \\ 0 
      \end{pmatrix}
      +
      \partial_\mu \lin(\mu_m) 
      \begin{pmatrix}
        V_{m}(\mu_m) \\ W_{m}(\mu_m) \\ 0 
      \end{pmatrix}
      =
      \frac {\mu_m^2}{J_0(\mu_m)} \wronsk_{1,m}'(\mu_m)
      \begin{pmatrix}
        0 \\ 0 \\ 
        \cos m\theta
      \end{pmatrix}.
    \end{align}
    Suppose now that $(v,w,\gamma) \in \X$ solves \eqref{eqn:nontrans}.
    Subtracting \eqref{eqn:aha}, we find
    \begin{align}
      \label{eqn:nosolns}
      \lin(\mu) 
      \begin{pmatrix}
        v-\partial_\mu V_m(\mu_m) \\ w-\partial_\mu W_m(\mu_m) \\ \gamma
      \end{pmatrix}
      =
      -\frac {\mu_m^2}{J_0(\mu_m)} \wronsk_{1,m}'(\mu_m)
      \begin{pmatrix}
        0 \\ 0 \\ 
        \cos m\theta
      \end{pmatrix}.
    \end{align}
    By Lemma~\ref{lem:bessel}, $\wronsk_{1,m}'(\mu_m) \ne 0$. 
    Arguing exactly as in the proof of Lemma~\ref{lem:ker}, we
    conclude that \eqref{eqn:nosolns} has no solutions, and hence that
    the same is true for \eqref{eqn:nontrans}.
  \end{proof}
\end{lemma}

\begin{proof}[Proof of Theorem~\ref{thm:main}]
  We apply Theorem~\ref{thm:cr} with $F=\F$, $X=\X$, $Y=\Y$, and
  $\mu^*=\mu_m$. The hypotheses \ref{cr:fred}, \ref{cr:ker},
  \ref{cr:trans} have been verified in Lemmas~\ref{lem:index},
  \ref{lem:ker}, and \ref{lem:transverse}, respectively. Thus there
  exists a family of solutions
  \begin{align*}
    \big(v(\varepsilon),w(\varepsilon),\gamma(\varepsilon),\mu(\varepsilon)\big) \in \X \by I
  \end{align*}
  to \eqref{eqn:newform}, parameterized by $\varepsilon \in (-\varepsilon_0,\varepsilon_0)$ and with the
  asymptotic expansions
  \begin{align*}
    v(\varepsilon) &= \varepsilon V_m(\mu_m) + O(\varepsilon^2), \\
    w(\varepsilon) &= \varepsilon W_m(\mu_m) + O(\varepsilon^2), \\
    \gamma(\varepsilon) &= O(\varepsilon^2), \\
    \mu(\varepsilon) &= \mu_m + O(\varepsilon) 
  \end{align*}
  as $\varepsilon \to 0$. Here for notational convenience we are suppressing the first arguments of
  $v,w,V_m,W_m$.
  The improved expansion 
  $\mu(\varepsilon) = \mu_m + O(\varepsilon^2)$
  for the parameter follows as usual from the symmetry of
  $\F$ under rotations by an angle $\pi/m$; see for instance
  \cite[discussion leading to (I.14.41)]{kielhofer:second}.
 
  These expansions in particular imply $\norm{w(\varepsilon)}_{C^{2+\alpha}} =
  O(\varepsilon)$ and so, perhaps after shrinking $\varepsilon_0 > 0$, the map
  $\phi(\varepsilon) = \mathrm{id} + w(\varepsilon)$ is a $C^{2+\alpha}$ diffeomorphism, and its
  image $\Omega(\varepsilon) \colonequals \phi(\D;\varepsilon)$ is a $C^{2+\alpha}$ domain. Thus, defining
  $u(\varepsilon) \in C^{2+\alpha}(\overline \D)$ using \eqref{eqn:vgamma}, i.e.
  \begin{align*}
    u(\varepsilon) \colonequals \tilde u(\varepsilon) \circ \phi(\varepsilon)^{-1}
    \quad 
    \text{where}
    \quad
    \tilde u(\varepsilon) \colonequals U(\mu(\varepsilon)) + v(\varepsilon) + \grad U(\mu(\varepsilon)) \cdot w(\varepsilon),
  \end{align*}
  and defining $c(\varepsilon) \colonequals U_r(1;\mu(\varepsilon)) - \gamma(\varepsilon)$ and $\lambda(\varepsilon) \colonequals (\mu(\varepsilon))^2$, we
  obtain the desired solutions of the original problem \eqref{eqn:og}
  with $b=1$. Here we must check, though, that we have chosen the
  correct sign for $c$, since in passing from \eqref{eqn:og} to
  \eqref{eqn:newform} we replaced \eqref{eqn:og:dyn} with
  \eqref{eqn:og:dyn2}. If necessary, further shrink $\varepsilon_0 > 0$ so that
  $c(\varepsilon) < 0$ for all $\varepsilon \in (-\varepsilon_0,\varepsilon_0)$. Inspecting the trivial solution
  at $\varepsilon = 0$, we see that the sign of $c(0)$ was indeed chosen
  correctly, and hence by a continuity argument that it is correct for
  all $\varepsilon \in (-\varepsilon_0,\varepsilon_0)$. 
 
  Finally, inserting the asymptotic expansions above into
  \eqref{eqn:vgamma} using the explicit formulas in
  \eqref{eqn:triv} and \eqref{eqn:VW} yields
  the claimed expansion
  \eqref{eqn:asym}, where here we have rescaled $\varepsilon$ by a factor of
  $J_m(\mu_m)J_0(\mu_m)/(\mu J_1(\mu_m))$ to simplify the form of \eqref{eqn:asym:F}.
\end{proof}

\section{Proof of Lemma~\ref{lem:bessel}}\label{sec:bessel}

This final section is devoted to the proof of Lemma~\ref{lem:bessel}.
First we introduce, for general integers $k,\ell \ge 0$, 
the Wronskians
\begin{align}
  \label{eqn:W}
  \wronsk_{k,\ell} := J_k J_\ell' - J_\ell J_k'.
\end{align}
As noted in \cite[Section~3.2]{palmai:interlacing}, an immediate consequence of
the differential equations \eqref{eqn:besselode} satisfied by
$J_k,J_\ell$ is that these Wronskians satisfy 
\begin{align}
  \label{eqn:dW}
  \frac d{d\mu} \big(\mu \wronsk_{k,\ell}(\mu)\big)  
  = \frac{\ell^2-k^2}\mu J_k(\mu)J_\ell(\mu).
\end{align}
This suggests that the
roots of $\wronsk_{k,\ell}$ are closely related to the
interlacing of the roots of $J_k,J_\ell$; see
\cite[Lemma~5]{palmai:interlacing}. Since Lemma~\ref{lem:bessel}
concerns the roots of $\wronsk_{1,m}$, we are therefore
particularly interested in understanding how the roots of
$J_1$ and $J_m$ interlace.

For integers $m \ge 0$, let $j_{m,1}<j_{m,2}<\cdots$ denote the positive
roots of $J_m$, all of which are simple. Importantly, $j_{m,n}$ is a
strictly increasing function of $m$ for each $n$. From the numerical
values $j_{3,1} \approx 6.3802$, $j_{1,2} \approx 7.0156$, and $j_{4,1} \approx 7.5883$,
we see that $j_{3,1} < j_{1,2} < j_{4,1}$, and so this monotonicity
implies
\begin{align}
  \label{eqn:besselroots}
  j_{1,2} < j_{m,1} \quad \text{for all } m \ge 4.
\end{align}
The inequality \eqref{eqn:besselroots} is the reason for the restriction $m\ge4$ in
Lemma~\ref{lem:bessel}. Indeed, the result is false for $m=0,1,2,3$.

\begin{figure}
  \includegraphics[scale=1.1]{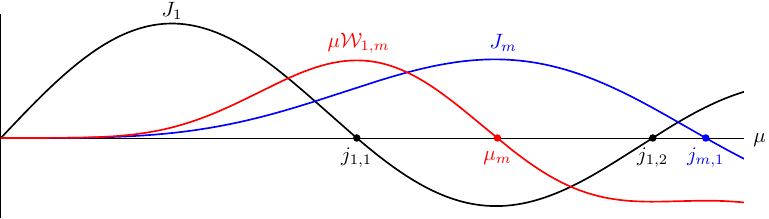} 
  \caption{
    Plots of $J_1(\mu),J_m(\mu),\mu\wronsk_{1,m}(\mu)$ for $m=4$. Their roots
    satisfy the inequalities $j_{1,1} < \mu_m < j_{1,2} < j_{m,1}$,
    and $\mu\wronsk_{1,m}$ is strictly decreasing on the interval
    $(j_{1,1},j_{1,2})$.
    \label{fig:wronsk}
  }
\end{figure}

\begin{proof}[Proof of Lemma~\ref{lem:bessel}]
  Fix an integer $m \ge 4$. The functions $J_1,J_m$ vanish at the origin
  and are positive for small positive arguments. This together with 
  \eqref{eqn:besselroots} gives the inequalities
  \begin{gather}
    \label{eqn:basicbessel}
    \begin{gathered}
      J_1 > 0 \ona (0,j_{1,1}),
      \quad 
      J_1 < 0 \ona (j_{1,1},j_{1,2}), 
      \quad 
      J_m > 0 \ona (0,j_{1,2}), \\
      J_1'(j_{1,1}) < 0 < J_1'(j_{1,2})
    \end{gathered}
  \end{gather}
  illustrated in Figure~\ref{fig:wronsk}.
  From the definition of $\wronsk_{1,m}$ we immediately conclude that
  \begin{align}
    \label{eqn:wronsk1}
    \wronsk_{1,m}(0) = 0,
    \quad
    \wronsk_{1,m}(j_{1,1}) > 0 > \wronsk_{1,m}(j_{1,2}),
  \end{align}
  while \eqref{eqn:dW} and \eqref{eqn:basicbessel} yield
  \begin{align}
    \label{eqn:wronsk2}
    (\mu\wronsk_{1,m})' > 0 \ona (0,j_{1,1}),
    \quad (\mu\wronsk_{1,m})' < 0 \ona (j_{1,1},j_{1,2}).
  \end{align}
  From \eqref{eqn:wronsk1} and \eqref{eqn:wronsk2} it is clear
  that $\wronsk_{1,m}$ has a root $\mu_m \in
  (j_{1,1},j_{1,2})$, that this root is simple, and that there are no
  other roots in the interval $(0,j_{1,2})$.

  Next we show that $\mu_k < \mu_m$ for $k > m \ge 4$. By
  \eqref{eqn:dW} and \eqref{eqn:basicbessel}, the function
  $\mu\wronsk_{k,m}$ is strictly decreasing on $(0,j_{1,2})$, so that in
  particular $\wronsk_{k,m}(\mu_m) < 0$. Using the algebraic
  identity
  \begin{align*}
    J_1 \wronsk_{k,m} - J_k \wronsk_{1,m} + J_m \wronsk_{1,k} = 0
  \end{align*}
  and the fact that $\wronsk_{1,m}(\mu_m)=0$, we therefore find 
  \begin{align*}
    \wronsk_{1,k}
     = -\frac{J_1}{J_m} \wronsk_{k,m} < 0
     \quad \ata \mu_m.
  \end{align*}
  Since \eqref{eqn:wronsk2} holds for $m=k$ and
  $\wronsk_{1,k}(\mu_k)=0$, we deduce that $\mu_k < \mu_m$ as desired.

  It remains to show that $\mu_m < j_{0,2}$. Since $j_{0,1} < j_{1,1} <
  j_{0,2} < j_{1,2}$, this will in particular imply that $J_0(\mu_m) <
  0$. As $\mu_m$ is a decreasing function of $m \ge 4$, it suffices to
  show $\mu_4 < j_{0,2}$, and arguing as above this is equivalent to the
  inequality $\wronsk_{1,4}(j_{0,2}) < 0$. Using a computer algebra
  system we numerically calculate $\wronsk_{1,4}(j_{0,2}) \approx -0.012148
  < 0$, and the proof is complete.
\end{proof}

\subsection*{Acknowledgments}
The author thanks Antonio J.~Fernández and David Ruiz for helpful
comments on an earlier draft of this paper.

\bibliographystyle{alpha}
\bibliography{master}

\end{document}